\documentclass[12pt,letterpaper,reqno]{article}
\usepackage{graphicx}
\usepackage{amsmath}
\usepackage{amsfonts}
\usepackage{amsthm}
\usepackage{amssymb}
\usepackage[T1]{fontenc}
\usepackage[toc,page]{appendix}
\usepackage{hyperref}
\usepackage{comment}

\theoremstyle{plain}
\newtheorem{theorem}{Theorem}[section]

\newtheorem{assumption}[theorem]{Assumption}

\newtheorem{proposition}[theorem]{Proposition}

\usepackage{ulem,soul,xcolor}

\newcommand{\mr}{\mathbb{R}}
\newcommand{\R}{\mathbb{R}}
\newcommand{\mc}{\mathcal}
\newcommand{\cA}{\mathcal{A}}

\newcommand{\cF}{\mathcal{F}}

\newcommand{\bP}{\mathbb{P}}
\newcommand{\bE}{\mathbb{E}}

\newcommand{\tr}{\mathrm{tr}}

\numberwithin{equation}{section}

\title{Non-explosion by Stratonovich noise for ODEs}
\author{
	Mario Maurelli\footnote{Dipartimento di Matematica `Federigo Enriques', Universit\`a degli Studi di Milano, via Saldini 50, 20133 Milano, Italy. E-mail address: {\tt mario.maurelli@unimi.it}}}
\date{}

\begin{document}

\maketitle

\begin{abstract}
We show that the addition of a suitable Stratonovich noise prevents the explosion for ODEs with drifts of super-linear growth, in dimension $d\ge 2$. We also show the existence of an invariant measure and the geometric ergodicity for the corresponding SDE.
\end{abstract}

\section{Introduction and main result}

An ODE on $\R^d$, with $d\ge 2$,
\begin{align}
dX = b(X)dt,\quad X_0=x_0,\label{eq:ODE}
\end{align}
with locally Lipschitz drift $b:\R^d\to\R^d$, can exhibit explosion in finite time: this is the case, for example, when
\begin{align*}
b(x) = |x|^{m-1}x,
\end{align*}
with $m>1$, as it can be checked computing the explicit solution. The main result of this paper is that the addition of a suitable Stratonovich noise can prevent the explosion. Precisely, we take the SDE on $\R^d$, with $d\ge 2$,
\begin{align}
dX = b(X)dt +\sigma(X) \circ dW,\quad X_0=x_0. \label{eq:SDE}
\end{align}
Here $x_0$ is given in $\R^d$, $W$ is a $d$-dimensional Brownian motion on a filtered probability space $(\Omega,\cA,(\cF_t)_t,\bP)$ (satisfying the standard assumption), $\circ$ denotes the Stratonovich integration and $b$ and $\sigma$ satisfy the following:
\begin{assumption}\label{hp}
We take $m>1$ and $\eta>(m-1)/2$. The drift $b:\R^d\to\R^d$ is locally Lipschitz and verifies, for some $C\ge0$,
\begin{align*}
|b(x)|\le C(1+|x|^m),\quad \forall x\in\R^d.
\end{align*}
The diffusion coefficient $\sigma:\R^d\to\R^{d\times d}$ is $C^1$ with locally Lipschitz derivative and satisfies, for some $R>0$,
\begin{align*}
\sigma(x) = |x|^{\eta+1} \left( I_d -(1+\frac{1}{\eta})\frac{xx^T}{|x|^2} \right), \quad \forall x\in B_R^c.
\end{align*}
\end{assumption}
We have used the notation $I_d$ for the $d\times d$-dimensional matrix, $B_R$ for the open ball of centre $0$ and radius $R$. Under Assumption \ref{hp}, the SDE admits a unique local strong solution.

Here is our main result:
\begin{theorem}\label{thm:main}
Under Assumption \ref{hp}, the SDE \eqref{eq:SDE} in $\R^d$, with $d\ge 2$, admits a global-in-time (unique, strong) solution.
\end{theorem}

The particular form of the noise is due to the following remark: under the transformation
\begin{align}
y = \phi(x) := |x|^{-\eta-1}x,\label{eq:transform}
\end{align}
the noise $\sigma(x) \circ dW$ simply becomes an additive noise. This remark reveals the idea behind the non-explosion: by applying the transformation $\phi$, the explosion becomes a passage through $0$, and such passage can be prevented, for $d\ge 2$, by an additive noise. We will give two proofs of this theorem: one exploits directly this idea of transforming explosion in passage through $0$, the other uses the Lyapunov function method.

The second proof (by the Lyapunov function method) may be useful in infinite dimensions and yields also important consequences for invariant distributions for the SDE. Indeed the Lyapunov function structure gives not only the existence of an invariant distribution for the SDE, but also a strong form of geometric ergodicity, whose bounds are independent of the initial conditions, as soon as the noise is non-degenerate on the whole space. These results on invariant measures are given and proved in Section \ref{sec:inv_meas}.

The first proof exploits crucially the fact that $d\ge 2$: indeed, in dimension $1$, a diffusion can hit $0$ in general. Actually, not only Theorem \ref{thm:main} does not hold for $d=1$, but, when the explosion occurs for the deterministic ODE \eqref{eq:ODE} in $d=1$ and the drift $b$ has a distinguished sign, then the explosion also happens $\bP$-a.s.~with a quite general Stratonovich noise. This fact is shown in Section \ref{sec:counterex}.

The possibility to use the noise to restore existence or uniqueness or stability has been widely explored. Many works deal with uniqueness and regularity for ODEs with irregular drifts, perturbed by an additive noise, transport equations, perturbed by a linear transport noise, and fluid dynamic models, see e.g.~\cite{Ver1980, FlaGubPri2010, CatGub2016, MohNilPro2015,Fla2015,BiaFla2020}. Concerning the stabilization by noise, we mention \cite{CraFla1998, BiaBloYan2016} among many others.

The non-explosion by noise is also widely studied. The work \cite{Sch1993} proves non-explosion and existence of an invariant measure for a two-dimensional example perturbed by an additive noise (possibly degenerate in one direction), by exhibiting a Lyapunov function for the SDE. The papers \cite{HerMat2015_1, HerMat2015_2, BirHerWeh2012, AthKolMat2012} prove similar results for other examples, in particular \cite{AthKolMat2012} proposes a meta-algorithm to find Lyapunov functions. In these examples, on one side the directions of explosion are isolated, hence the idea behind the non-explosion phenomenon is that the noise moves the solution out of those directions. On the other side only additive noise is allowed, as often dictated by applications (for example, the equation in \cite{HerMat2015_1} comes from models in turbulent transport). The difficulty in proving non-explosion in such examples is often to find an appropriate, anisotropic Lyapunov function, in order to deal with different (explosive or not) behaviours in different regions of the space (see e.g.~\cite[Section 2.1]{AthKolMat2012}).

In the present paper instead a different viewpoint is considered: on one side every direction is possibly explosive, on the other side we search for a multiplicative noise which could avoid explosion. This research direction has also been studied in a number of works, see \cite{MaoMarRen2002, WuHu2009, LiuShe2012} among many others, mostly exploiting the Lyapunov function method. The closest results to ours seem \cite{AppMaoRod2008,Gar1988} and \cite{Has1960}. The results \cite[Proposition 3.3]{AppMaoRod2008} and \cite[Example 5.4]{Gar1988} show non-explosion by noise for a large class of drifts and diffusion coefficients, using respectively power-type and logarithmic Lyapunov functions. In particular, by writing our SDE \eqref{eq:SDE} in It\^o form, outside the ball $B_R$,
\begin{align*}
dX &= b(X)dt -\frac12 (1+\frac{1}{\eta})(d-1-\frac{1}{\eta})|X|^{2\eta}X dt\\
&\quad +|X|^{\eta+1} \left( I_d -(1+\frac{1}{\eta})\frac{XX^T}{|X|^2} \right) dW,
\end{align*}
we can recognize that our SDE falls in the class of \cite{AppMaoRod2008,Gar1988} for $d\ge 3$. The novelties here, with respect to \cite{AppMaoRod2008,Gar1988}, are the idea behind the first proof, based on the Stratonovich noise and the transformation in \eqref{eq:transform}, and the use of $(\log |x|)^\alpha$, for some $\alpha<1$, as Lyapunov function, to deal with the case $d=2$ (in \cite[Remark 3]{ChoKha2014} the use of a similar Lyapunov function is suggested). The Stratonovich noise arises as limit of smooth noises and is also widely used in SPDEs. In the SPDE context, we point out the papers \cite{GasGes2019, GesSmi2019}, among others, which also take non-linear noises to show regularization properties. We also mention the recent paper \cite{KCSW2019}, which shows non-explosion for a Hamiltonian ODE perturbed by an additive noise and a suitable drift term, which preserves the Hamiltonian structure.

The paper \cite{Has1960} establishes a criterion, sometimes called Khasminskii's test (generalization of Feller's test), for non-explosion for a multi-dimensional diffusion. According to this criterion (in the version in \cite[Section 4.5]{McK1969} and \cite[Theorem 10.2.3]{StrVar2006}), explosion is avoided if a certain integral condition holds for suitable radial functions of the coefficients of the SDE (in It\^o form). For our example, by the isotropic form of Assumptions \ref{hp}, it is not difficult to show such integral condition; hence we could prove our Theorem \ref{thm:main} also by an application of Khasminskii's test. With respect to this possible approach, the differences here are again the idea behind the first proof and the use of an explicit Lyapunov function in the second proof. As already mentioned, such Lyapunov function has important consequences at the level of invariant measures, see Section \ref{sec:inv_meas}.


Once non-explosion from any fixed $x_0$ is established, one could ask about finer properties, like global existence of a stochastic flow solving the SDE \eqref{eq:SDE}. We expect the answer to be negative. Indeed, in \cite{CarElw1983}, the authors construct an SDE by applying a similar transformation to bring $\infty$ into $0$ and vice versa and show the lack of a stochastic flow solution of that SDE; see also \cite{LiSch2011} and \cite{LeiSch2017} for a similar phenomenon respectively for a drift-less SDE with bounded, smooth coefficients and for the example in \cite{HerMat2015_1}.


\section{First proof}

In the first proof, we apply the transformation $Y=\phi(X)$ to the SDE \eqref{eq:SDE}; by the specific form of $\sigma$ we get an SDE for $Y$ with irregular drift $g(Y)$ and additive noise. The conditions on $m$ and $\alpha$ in Assumption \ref{hp} guarantee that this SDE for $Y$ admits a unique solution, whose law is equivalent to the Wiener measure (by Girsanov theorem). Since, for $d\ge 2$, the Wiener measure does not see the $0$ $\bP$-a.s.,~we conclude that $Y$ does not hit $0$ and hence $X$ does not explode $\bP$-a.s..

\begin{proof}[First proof]
Let $(X_t)_{t\in[0,\tau)}$ be the local maximal solution to the SDE \eqref{eq:SDE}. Recall that $\tau$ is the explosion time of $X$, or equivalently, on $\{\tau<\infty\}$, there holds $\lim_{t\nearrow\tau}|X_t| =\infty$ $\bP$-a.s.~(see e.g.~\cite[Corollary 6.2]{Elw1982}). We have to show that $\tau=\infty$ $\bP$-a.s.. For all $t$ such that $X_t\neq 0$, we let $Y_t=\phi(X_t)$, where $\phi$ is defined as in \eqref{eq:transform}. More precisely, to avoid the times when $X$ hits $0$ (and so $Y$ is not defined), we will show first that, $\bP$-a.s.,~$X$ enters $B_R$ before exploding, then we will use this fact to conclude the proof of non-explosion.

\textbf{First part}: We use the notation $\bP^{x_0}$ to keep track of the initial condition $x_0$. We define $\tau^{0,R} = \inf\{t\ge 0\mid X_t\in B_R\}$. We will show that
\begin{align}
\bP^{x_0}\{\tau\ge \tau^{0,R}\}=1,\quad \forall x_0 \in B_{R+1}^c.\label{eq:returning}
\end{align}
We start applying It\^o formula to $Y_t=\phi(X_t)$ for $t< \tau\wedge \tau_{0,R}$ (this is possible because $\phi$ is smooth on $B_R^c$). Assumption \ref{hp} on $\sigma$ gives
\begin{align*}
D\phi(x) = |x|^{-\eta-1}(I_d -(\eta+1)\frac{xx^T}{|x|^2}) = \sigma(x)^{-1},\quad \forall x\in B_R^c.
\end{align*}
Note that
\begin{align}
\begin{aligned}\label{eq:notation_stop_time}
&\tau=\rho^Y :=\inf\{t\ge 0\mid \lim_{s\nearrow t}Y_s=0\},\\
&\tau^{0,R} =\rho^{Y,R^{-\eta}} :=\{t\ge 0\mid Y_t\in B_{R^{-\eta}}^c\}.
\end{aligned}
\end{align}
Hence the following SDE holds for $Y$ on $[0,\rho^Y\wedge \rho^{Y,R^{-\eta}})$:
\begin{align}
\begin{aligned}\label{eq:SDE_Y}
dY &= |Y|^{(\eta+1)/\eta}(I_d -(\eta+1)\frac{YY^T}{|Y|^2}) b(|Y|^{-1/\eta-1}Y) dt +dW\\
&=: g(Y) dt +dW,\\
Y_0 &=\phi(x_0) \in \bar{B}_{(R+1)^{-\eta}}\setminus\{0\}.
\end{aligned}
\end{align}
Since $Y$ lives in $B_{R^{-\eta}}$, we can set $g=0$ on $\bar{B}_{R^{-\eta}}^c$. Since $b$ is locally Lipschitz, $g$ is locally Lipschitz on $\bar{B}_{R^{-\eta}}\setminus\{0\}$. Moreover, Assumption \ref{hp} for $b$ gives, for some constant $C$,
\begin{align*}
|g(y)|\le C|y|^{(\eta+1)/\eta} \cdot (|y|^{-1/\eta})^m = C|y|^{(\eta+1-m)/\eta}, \quad \forall y\in B_{R^{-\eta}}.
\end{align*}
By Assumption \ref{hp} on $m$ and $\eta$, we have $(\eta+1-m)/\eta >-1$, so the drift $g$ is in $L^p(\R^d)$ for some $p>d$ (and $d\ge 2$). We are now in the position to apply \cite[Theorem 1, Corollary 16]{FedFla2011}: the SDE \eqref{eq:SDE_Y} admits a global (strong) solution $\tilde{Y}$ whose law is equivalent to the $d$-dimensional Wiener measure starting from $\phi(x_0)$. But the SDE \eqref{eq:SDE_Y} admits also a unique strong solution before exiting $\bar{B}_{R^{-\eta}}\setminus\{0\}$, by the local Lipschitz property of $g$, and hence $\tilde{Y}=Y$ on $[0,\rho^Y\wedge \rho^{Y,R^{-\eta}})$.
In particular,
\begin{align*}
\text{Law}(Y\mid_{[0,\rho^Y\wedge \rho^{Y,R^{-\eta}})})\text{ and } \text{Law}(W^{x_0}\mid_{[0,\rho^{W^{x_0}}\wedge \rho^{W^{x_0},R^{-\eta}})})
\end{align*}
are equivalent, where $W^{x_0}:=W+\phi(x_0)$ and $\rho^{W^{x_0}}$ and $\rho^{W^{x_0},R^{-\eta}}$ are defined for $W^{x_0}$ as in \eqref{eq:notation_stop_time}. Now, for $d\ge 2$, for any $x_0$, $W+\phi(x_0)$ does not hit $0$ with probability $1$, that is $\rho^{W^{x_0}}=\infty$ $\bP$-a.s.~(see e.g.~\cite[Chapter V, Proposition 2.7]{RevYor1999}). Hence we have $\rho^Y\ge \rho^{Y,R^{-\eta}}$ $\bP$-a.s.~and so $\tau\ge \tau^{0,R}$ $\bP$-a.s.,~that is \eqref{eq:returning}.

\textbf{Second part}: We use a standard argument. We recall that $\tau^{0,R} = \inf\{t\ge 0\mid X_t\in B_R\}$ and we define recursively, for $i$ nonnegative integer,
\begin{align*}
\tau^{i+1,R+1} = \inf\{t>\tau^{i,R} \mid X_t\notin B_{R+1}\},
\tau^{i+1,R} = \inf\{t>\tau^{i+1,R+1} \mid X_t\in B_R\},
\end{align*}
that is the $(i+1)$-th exit time from $B_{R+1}$ and the $(i+1)$-th hitting time of $B_R$. The property \eqref{eq:returning} and the strong Markov property of $X$ imply, by induction taking $x_0 = X_{\tau^{i,R+1}}$ at each step, that $\tau\ge \tau^{i,R}$ $\bP$-a.s.~for every $i$. On the other hand, since $\lim_{t\nearrow\tau}|X_t| =\infty$ $\bP$-a.s.~on $\{\tau<\infty\}$, then $\bP$-a.s.,~the sequence $\tau^{i,R}$ has no finite accumulation point, that is $\sup_i \tau^{i,R} =\infty$ $\bP$-a.s.. It follows that $\tau=\infty$ $\bP$-a.s.. The proof is complete.
\end{proof}

\section{Second proof}

In the second proof, we show that $(\log|x|)^\alpha$, with $0<\alpha<1$, is morally a Lyapunov function for the SDE \eqref{eq:SDE}. This argument not only implies non-explosion, but has also consequences on invariant measures for the SDE, as we will see in the next section.

\begin{proof}[Second proof]
We fix $0<\alpha<1$ and take a $C^2$ function $V:\R^d\rightarrow \R$ such that, for a constant $a>0$,
\begin{align*}
&V(x) = (\log|x|)^\alpha,\quad \forall x \in B_{2\vee R}^c,\\
&V(x)\ge a,\quad \forall x\in \R^d,
\end{align*}
and we show that such $V$ is a Lyapunov function for the SDE \eqref{eq:SDE}, that is
\begin{itemize}
\item $V$ is nonnegative and $\inf_{|x|>r} V(x)$ tends to $\infty$ as $r\rightarrow\infty$ and
\item $LV\le cV$ on $\R^d$ for a constant $c>0$, where $L$ is the generator of the SDE \eqref{eq:SDE}.
\end{itemize}
The first condition is clearly satisfied. For the second condition, we write the SDE \eqref{eq:SDE} in It\^o form:
\begin{align*}
dX = \tilde b(X)dt +\sigma(X) dW,
\end{align*}
where $\tilde b$ is locally Lipschitz and
\begin{align*}
\tilde b(x) = b(x) -\frac12 (1+\frac{1}{\eta})(d-1-\frac{1}{\eta})|x|^{2\eta}x, \quad \forall x\in B_R^c.
\end{align*}
This can be verified through a tedious computation of the It\^o-Stratonovich correction of \eqref{eq:SDE} or applying It\^o formula (in It\^o form) to $X=\phi^{-1}(Y)$, where $Y$ satisfies \eqref{eq:SDE_Y}. For $x$ in $B_{2\vee R}^c$, we have then
\begin{align*}
&\nabla V(x)= \alpha (\log|x|)^{\alpha-1}|x|^{-2} x,\\
&D^2V(x) = \alpha (\log|x|)^{\alpha-1}|x|^{-2} \left( I_d -\left(2+(1-\alpha)(\log|x|)^{-1}\right)\frac{xx^T}{|x|^2} \right),\\
&\sigma(x)\sigma(x)^T = |x|^{2\eta+2} \left(I_d- (1+\frac{1}{\eta})\frac{xx^T}{|x|^2}\right)^2 = |x|^{2\eta+2} \left(I_d +(-1+\frac{1}{\eta^2})\frac{xx^T}{|x|^2}\right),
\end{align*}
and so
\begin{align}
LV(x) &= \tilde b(x)\cdot \nabla V(x) +\frac12 \tr[\sigma(x)\sigma(x)^TD^2V(x)]\nonumber\\
&= \alpha (\log|x|)^{\alpha-1} b(x) \cdot |x|^{-2} x -\frac12 \alpha (\log|x|)^{\alpha-1} |x|^{2\eta} (1+\frac{1}{\eta})(d-1-\frac{1}{\eta})\nonumber\\
&\quad +\frac12 \alpha (\log|x|)^{\alpha-1} |x|^{2\eta} \tr\left[ I_d -(1+\frac{1}{\eta^2}+\frac{1-\alpha}{\eta^2\log|x|})\frac{xx^T}{|x|^2} \right]\nonumber\\
&= \alpha (\log|x|)^{\alpha-1} b(x) \cdot |x|^{-2} x -\frac12 \alpha (\log|x|)^{\alpha-1} |x|^{2\eta} \left(\frac{d-2}{\eta} +\frac{1-\alpha}{\eta^2\log|x|}\right)\nonumber\\
&\le (\log|x|)^{\alpha-1} \left( c_1|x|^{m-1} -c_2\frac{1}{\log|x|}|x|^{2\eta} \right) \label{eq:Lyapunov}
\end{align}
for some constants $c_1,c_2>0$. By the condition $2\eta >m-1$ in Assumption \ref{hp}, there exists $r>0$ such that $LV(x)$ is negative for all $x$ outside $B_r$. Therefore, since $LV$ is locally bounded and $V\ge a>0$, the condition $LV\le cV$ is satisfied on $\R^d$ for a suitable $c$ and so $V$ is a Lyapunov function. Hence, by \cite[Theorem 3.5]{Kha2012}, there exists a global solution to the SDE \eqref{eq:SDE}. The proof is complete.
\end{proof}

\section{Invariant measures and geometric ergodicity}\label{sec:inv_meas}

In this section we exploit the Lyapunov function structure to show existence and uniqueness of an invariant probability measure for the SDE and a strong form of geometric ergodicity. We remind that a probability measure $\mu$ on $\mr^d$ is invariant for the SDE \eqref{eq:SDE} if it is invariant under the Markov semigroup associated with \eqref{eq:SDE}, or equivalently if there exists $X^\mu$ solution to \eqref{eq:SDE} (with random initial condition) such that $X^\mu_t$ has law $\mu$ for any $t\ge 0$.

The existence of an invariant (probability) measure for the SDE \eqref{eq:SDE} is a consequence of the existence of a Lyapunov function:

\begin{proposition}
Under Assumption \ref{hp}, the SDE \eqref{eq:SDE} in $\R^d$, with $d\ge 2$, admits an invariant probability measure.
\end{proposition}

\begin{proof}
The inequality \eqref{eq:Lyapunov} implies that
\begin{align*}
\sup_{|x|>r}LV(x) \rightarrow -\infty \quad \text{as }r\rightarrow \infty.
\end{align*}
Hence the result follows from \cite[Theorem 3.7]{Kha2012}.
\end{proof}

Actually the inequality \eqref{eq:Lyapunov} gives more than the Lyapunov function property and the existence of an invariant measure. Indeed, in the language of \cite{AthKolMat2012}, \eqref{eq:Lyapunov} shows that $V$ is a super-Lyapunov function. As a consequence, if the noise is non-degenerate on $B_R$, we get geometric ergodicity in the total variation norm, whose bounds are independent of the initial conditions (see the discussion at the end of this section).

In the following, we use the notation $(P_t)_t$ for the Markov semigroup associated with the SDE \eqref{eq:SDE} and, given a probability measure $\mu$ on $\mr^d$, we use $P_t^*\mu$ for the transformation of $\mu$ under the semigroup $P_t$, namely
\begin{align*}
&P_tf(x_0) = \bE^{x_0}[f(X_t)],\quad x_0\in \mr^d,f\in C_b(\mr^d),\\
&\int_{\mr^d} f(x)P_t^*\mu(dx) = \int_{\mr^d} P_tf(x)\mu(dx),\quad f\in C_b(\mr^d).
\end{align*}
$\mc{P}(\mr^d)$ denotes the set of probability measures on $\mr^d$ and $d_{TV}$ denotes the total variation distance between measures. We also introduce the following weighted total variation distance:
\begin{align*}
d_1(\mu_1,\mu_2) = \sup_{\varphi\in C(\mr^d),\,\|\varphi/(1+V)\|_{\infty}\le 1} \int_{\mr^d} \varphi d(\mu_1-\mu_2),\quad \mu_1,\mu_2\in \mc{P}(\mr^d).
\end{align*}
It is easy to check that $d_{TV}\le d_1$.

\begin{theorem}\label{thm:geom_erg}
Assume Assumption \ref{hp}. Assume also that, for some constant $\lambda>0$,
\begin{align}
\sigma(x)\sigma(x)^T \ge \lambda I_d,\quad \forall x\in B_R. \label{eq:elliptic}
\end{align}
Then the invariant distribution for the SDE \eqref{eq:SDE} on $\mr^d$, with $d\ge 2$, is unique and the following geometric ergodicity property holds in the total variation distance: there exist $C>0$, $\eta>0$ such that,
\begin{align}
d_1(P_t^*\mu_1,P_t^*\mu_2) \le Ce^{-\eta t}d_{TV}(\mu_1,\mu_2), \quad \forall t\ge 0, \quad \forall \mu_1,\mu_2 \in \mc{P}(\mr^d).\label{eq:exp_conv} 
\end{align}
In particular, there holds
\begin{align*}
d_{TV}(P_t^*\mu_1,P_t^*\mu_2) \le Ce^{-\eta t}d_{TV}(\mu_1,\mu_2), \quad \forall t\ge 0, \quad \forall \mu_1,\mu_2 \in \mc{P}(\mr^d).
\end{align*}
\end{theorem}

The proof is morally a consequence of \cite[Lemma 6.1, Proposition 6.2, Theorem 12.1]{AthKolMat2012}; these results are given in \cite{AthKolMat2012} for an SDE on $\mr^2$, but the arguments in the proofs hold for a general SDE on $\mr^d$ (for any $d$) with locally Lipschitz drift and diffusion coefficients. The following result follows immediately from \cite[Lemma 6.1, Proposition 6.2, Theorem 12.1]{AthKolMat2012}:

\begin{theorem}\label{thm:superLyap_minor}
Assume that $V:\mr^d\rightarrow \mr$ is a super-Lyapunov function for the SDE \eqref{eq:SDE}, namely a $C^2$, strictly positive function such that $\lim_{|x|\rightarrow\infty}V(x) = \infty$ and, for suitable constants $\gamma>1$, $c,d>0$,
\begin{align}
LV(x) \le -cV(x)^\gamma +d,\quad \forall x\in\mr^d.\label{eq:super_Lyap}
\end{align}
Assume also the following minorization condition: there exist $T>0$, $\alpha>0$, $r>0$ and a probability measure $\nu$ on $\mr^d$ such that
\begin{align}
\inf_{z\in\mr^d,|z|\le r} P_T^*\delta_z(B) \ge \alpha\nu(B), \quad \forall B\in \mc{B}(\mr^d),\label{eq:minor}
\end{align}
with $r$ satisfying
\begin{align*}
r> K_T:=\max \{ (2d/c)^{1/\gamma}, (c(\gamma-1)T/2)^{-1/(\gamma-1)} \}.
\end{align*}
Then the conclusions of Theorem \ref{thm:geom_erg} hold.
\end{theorem}

\begin{proof}[Proof of Theorem \ref{thm:geom_erg}]
By Theorem \ref{thm:superLyap_minor}, it is enough to show the super-Lyapunov function property \eqref{eq:super_Lyap} and the minorization condition \eqref{eq:minor}.

The property \eqref{eq:super_Lyap} follows immediately from \eqref{eq:Lyapunov} (and positivity of $V$), by the condition $m-1>2\eta$ from Assumption \ref{hp}.

For the minorization condition, we note that, by the explicit computation of $\sigma(x)\sigma(x)^T$ on $B_R^c$, the uniform ellipticity condition \ref{eq:elliptic} holds actually on the whole $\mr^d$. Hence a classical Gaussian lower bound holds (see e.g.~\cite[Theorem 7]{AroSer1967}): for every $r>0$, there exist $c_1,c_2>0$ such that, for all $z\in \bar{B}_r$, for all $t>0$ sufficiently small,
\begin{align*}
P_t^*\delta_z \ge c_1 t^{-d/2} e^{-c_2|x-z|^2/t} 1_{|x-z|<2r}dx.
\end{align*}
Hence the minorization condition \eqref{eq:minor} holds for any $r>0$, taking as $\nu$ the normalized Lebesgue measure on $B_r$. The proof is complete.
\end{proof}

The main strength of Theorem \ref{thm:geom_erg} is in the exponential convergence to the invariant measure and in the use of the total variation distance, instead of the weighted $d_1$ distance, in the right-hand side of \eqref{eq:exp_conv}. It is known that condition \eqref{eq:super_Lyap} with $\gamma=1$ and a mixing condition guarantee exponential convergence to the invariant measure, but only in the $d_1$ norm, see e.g.~\cite{MeyTwe2009,HaiMat2011}: roughly speaking, this gives, calling $\bar{\mu}$ the invariant measure for the SDE,
\begin{align*}
d_{TV}(P_t^*\delta_z,\bar{\mu}) \le Ce^{-\eta t}(1+V(z)).
\end{align*}
When condition \eqref{eq:super_Lyap} is satisfied only with $\gamma<1$, then polynomial convergence holds instead of exponential convergence (see e.g.~\cite{Ver1997,Ver1999,DouForGui2009}). As discovered in \cite{AthKolMat2012}, the existence of a super-Lyapunov function gives exponential convergence in the total variation norm: in particular, one has
\begin{align*}
d_{TV}(P_t^*\delta_z,\bar{\mu}) \le Ce^{-\eta t},
\end{align*}
without any weight on $z$ in the right-hand side; in this sense, the bounds in Theorem \ref{thm:geom_erg} are independent of the initial conditions.

\section{Counterexample in one dimension}\label{sec:counterex}

In this section we show a counterexample to Theorem \ref{thm:main} when $d=1$; actually, we show that an explosive ODE does also explode under a very general Stratonovich noise. The fact that noise cannot avoid explosion when $d=1$ has been already proved in \cite{Sch1995} in the case of additive noise.

We consider the SDE \eqref{eq:SDE} on $\mr$, with $b$ locally Lipschitz and $\sigma$ $C^1$ with locally Lipschitz derivative. We assume that $b$ and $\sigma$ are always strictly positive. Under these conditions, it is well-known that, for every $x_0$, the solution to the ODE \eqref{eq:ODE} on $\mr$ admits the explicit representation
\begin{align*}
x(t) = B^{-1}(B(x_0)+t),
\end{align*}
where $B(x) = \int_0^x 1/b(y) dy$. In particular, explosion holds for the ODE \eqref{eq:ODE} if and only if
\begin{align}
\int_0^\infty \frac{1}{b(z)} dz < \infty.\label{eq:expl_1d}
\end{align}

\begin{proposition}
Assume that $b:\mr\rightarrow\mr$ is locally Lipschitz and that $\sigma:\mr\rightarrow\mr$ is $C^1$ with locally Lipschitz derivative; assume that $b$ and $\sigma$ are always strictly positive. Assume condition \eqref{eq:expl_1d}. Then, in dimension $d=1$, for every $x_0$ in $\mr$, the SDE \eqref{eq:SDE} explodes $\bP$-a.s..
\end{proposition}

The idea of the proof is not far from the first proof of Theorem \ref{thm:main}: we use the transformation $\phi$ with $D\phi(x)=\sigma(x)^{-1}$.

\begin{proof}
We let
\begin{align*}
\phi(x) = \int_0^x \sigma(y)^{-1} dy,\quad \phi(-\infty)<x<\phi(\infty).
\end{align*}
By It\^o formula, $Y=\phi(X)$ satisfies the following SDE:
\begin{align*}
dY = A(Y)dt +dW,\quad Y_0=\phi(x_0),
\end{align*}
where
\begin{align*}
A(y) = \frac{b(\phi^{-1}(y))}{\sigma(\phi^{-1}(y))}.
\end{align*}
Hence it is enough to show that $Y$ hits $\phi(\infty)$ in finite time $\bP$-a.s.. We distinguish two cases: $\phi(\infty)<\infty$ or $=\infty$. If $\phi(\infty)<\infty$, since $Y_t\ge \phi(x_0)+W_t$, then $Y$ hits $\phi(\infty)$ in finite time $\bP$-a.s.. If $\phi(\infty)=\infty$, we have, by condition \eqref{eq:expl_1d} and a change of variable,
\begin{align*}
\int_0^\infty \frac{1}{A(y)} dy = \int_0^\infty \frac{1}{b(y)} dx <\infty.
\end{align*}
By \cite[Corollary 2]{Sch1995}, we get
\begin{align*}
2\int_0^\infty \int_0^\infty \exp\left[-2\int_y^{y+z} A(u)du\right] dzdy \le \int_0^\infty \frac{1}{A(y)} dy <\infty.
\end{align*}
Hence we can apply Feller's test for explosion, in the form of \cite[Proposition 5.32 (point (ii))]{KarShr1991}, and get that $Y$ hits $\infty=\phi(\infty)$ in finite time $\bP$-a.s.. The proof is complete.
\end{proof}

\section*{Acknowledgements}

I thank two anonymous referees for their valuable comments and suggestions, in particular for suggesting most of the results and proofs in Sections \ref{sec:inv_meas} and \ref{sec:counterex}. I thank also Franco Flandoli for useful discussions. I acknowledge support from the Hausdorff Research Institute for Mathematics in Bonn under the Junior Trimester Program `Randomness, PDEs and Nonlinear Fluctuations' and from the Italian Ministry of Education, University and Research under the project PRIN 2015233N54\_002 `Deterministic and stochastic evolution equations'.

\bibstyle{alpha.bst}
\bibliographystyle{alpha}
\bibliography{my_bib3}

\end{document}